\documentclass[sn-vancouver,Numbered]{sn-jnl}
\usepackage{graphicx}%
\usepackage{multirow}%
\usepackage{amsmath,amssymb,amsfonts}%
\usepackage{amsthm}%
\usepackage{mathrsfs}%
\usepackage[title]{appendix}%
\usepackage{xcolor}%
\usepackage{textcomp}%
\usepackage{manyfoot}%
\usepackage{booktabs}%
\usepackage{algorithm}%
\usepackage{algorithmicx}%
\usepackage{algpseudocode}%
\usepackage{listings}%
\usepackage{bbm}
\usepackage{subfig}

\newcommand{\R}{\mathbb{R}}
\newcommand{\N}{\mathbf{N}}

\newcommand{\V}{\mathcal{V}}
\newtheorem{theorem}{Theorem}
\newtheorem{lemma}{Lemma}
\newtheorem{corollary}{Corollary}
\newtheorem{remark}{Remark}

\newtheorem{hypo}{Hypothesis}

\begin{document}
\title[Stability analysis]{Stability analysis of an inverse coefficients problem in a  system of partial differential equations}
\author*[1]{\fnm{Houcine} \sur{Meftahi}}\email{houcine.meftahi@enit.utm.tn}
\equalcont{These authors contributed equally to this work.}
\author[2]{\fnm{Chayma} \sur{Nssibi}}\email{chayma.nssibi@enit.utm.tn}
\equalcont{These authors contributed equally to this work.}
\affil[1]{\orgdiv{University of Jendouba/ISIK and ENIT-Lamsin,Tunisia}}
\affil[2]{\orgdiv{University of Tunis Elmanar, ENIT-LAMSIN, B.P. 37, 1002 Tunis}}

\abstract{
In this study, we address the inverse problem of recovering the Lamé parameters 
($\lambda, \mu$) and the density $\rho$ of a medium from the Neumann-to-Dirichlet map   for any dimension $d\geq 2$. This inverse problem finds its motivation in the reconstruction of mechanical properties of tissues in  medical diagnostics. We first assume that the Lamé parameters ($\lambda, \mu$) are know and we look for the inverse problem of recovering the density $\rho$. In this context, we derive a constrcutive Lipschitz stability estimate in terms of the  Neumann to Dirichlet map in the case of piecewise constant parameters. Then, we look for the inverse problem of recovering $\lambda$, $\mu$ and $\rho$ simultameousely.
 We establish Lipschitz stability  estimate, provided that the parameters $\lambda$, 
$\mu$ and $\rho$ have upper and lower bounds and belong to a known finite-dimensional subspace.
The proofs hinge on monotonicity relations between the parameters and the Neumann-to-Dirichlet operator, coupled with the techniques of localized potentials.}

\maketitle
\section{Introduction}
In this article, we explore the inverse problem of reconstructing the spatially varying parameters $\lambda(x)$, $\mu(x)$, and $\rho(x)$ in the context of the following elliptic system:
\begin{equation}\label{prob1}
-\nabla\cdot \left(\lambda (\nabla\cdot u) I_d+2\mu\nabla^s u\right)+ \rho  u = 0 \text{ in } \Omega,
\end{equation}
where, $\nabla^s u$ denotes a symmetric matrix associated with the displacement field $u$, defined as $\left(\nabla u+\nabla u^T\right)/2$. Additionally, $I_d$ represents the identity matrix of size $d\times d$. The coefficients $\lambda$, $\mu$,  $\rho$  are distinctly referred to as the compression modulus, shear modulus,  and density modulus, respectively.

Equation \eqref{prob1} can be interpreted  as  an approximate model for small deformations in biological tissues, where the term 
$\rho u$ represents a body force proportional to the displacement,  known as the restoration force  \cite{villa2021mechanical}.

 The primary impetus behind tackling this inverse problem lies in its applications to non-destructive testing of elastic structures, such as those encountered in geophysics \cite{aki1980quantitative} and medical imaging \cite{gennisson2013ultrasound, parker2010imaging, doyley2012model}. This is particularly significant in the context of tumor detection \cite{yuan2014application}.

This inverse problem exhibits similarities to the inverse conductivity problem \cite{alessandrini1988stable, harrach2019uniqueness, harrach2019global} and is recognized for its severe ill-posedness, where small measurement errors can result in significantly different outcomes. Unfortunately, the mathematical techniques applied to study the conductivity inverse problem cannot be directly extended to the present problem. This limitation stems from dealing with an elliptic system featuring multiple parameters as functions of space, rather than a scalar equation with a single parameter. Consequently, due to these inherent difficulties, there are only a limited number of known results regarding uniqueness and stability.

Let's explore some relevant findings related to our problem. The task of determining an unknown pair $\gamma = (\lambda, \mu)$ of Lamé parameters in the context of the isotropic linear elasticity system,
was initially addressed by Ikehata \cite{ikehata1990inversion}. In this study, the author established a uniqueness result for the case of constant $\gamma$. The proof relies on the first-order approximation of the Dirichlet-to-Neumann map $\Lambda(\gamma+h)$ at $\gamma$.

In the two-dimensional scenario, Imanuvilov and Yamamoto \cite{imanuvilov2015global} demonstrated a uniqueness result for $C^{10}$ Lamé parameters. For three dimensions, Nakamura and Uhlmann \cite{nakamura1994global} established uniqueness assuming that the Lamé parameters are $C^\infty$ and that $\mu$ is close to a positive constant.

Beretta et al. \cite{Bretta} contributed a uniqueness and Lipschitz stability result, assuming that the Lamé parameters are piecewise constant and the interfaces of discontinuity contain flat parts. Further results on uniqueness and stability can be found in \cite{akamatsu1991identification,nakamura1995inverse,eskin2002inverse,nakamura1993identification,nakamura2003global}.

The proof of uniqueness and stability in the aforementioned inverse coefficient problems typically relies on intricate methodologies, involving a combination of singular solutions and unique continuation estimates. Further insights into these techniques can be found in  \cite{beretta2011lipschitz, alessandrini2005lipschitz, gaburro2015lipschitz, alessandrini2018lipschitz, beretta2013lipschitz, alberti2019calderon}, addressing related problems. For the numerical inversion methods we refer the reader to \cite{ammari2021direct, ammari2015mathematical,ammari2008method} and the references therin.

In contrast, for our specific problem, we employ simpler techniques based  in a combination of monotonicity properties and the existence of localized potentials.

The principle challenge in our inverse coefficients problem stems from its multiparametric nature. In contrast, for the inverse optical tomography problem,
 characterized by the scalar equation
\[
-\nabla\cdot (\sigma\nabla u)+ q u= 0,
\]  
Arridge and Lionheart \cite{arridge1998nonuniqueness} have shown a non-uniqueness result regarding the simultaneous recovery of a general $\sigma$ and $q$. However, in cases where the diffusion coefficient $\sigma$ is piecewise constant and the absorption term $q$ is piecewise analytic, uniqueness and Lipschitz stability results have been established in  \cite{harrach2012simultaneous,meftahi2020uniqueness}.

In this study, we focus on the elliptic system and establish Lipschitz stability estimates. It's important to note that for arbitrary Lamé parameters $(\lambda, \mu)$ and density $\rho$, Lipschitz stability is not guaranteed \cite{mandache2001exponential}. However, by considering parameters within a specific finite-dimensional subspace, we demonstrate Lipschitz stability estimates using the tools of monotonicity and localized potentials. In particular we derived a contructive Lipschitz stability estimate for the inverse problem of recovering the density $\rho$  in terms of the  Neumann to Dirichlet map, in the case of piecewise constant parameters.

It's worth highlighting that the result obtained here does not constitute a direct extension of the findings in \cite{meftahi2020uniqueness,meftahi2023elastic,eberle2021lipschitz} because it requires the verification of essential, non-trivial assumptions to guarantee the effective application of localized potentials.

The concept of employing monotonicity properties and localized potentials techniques has yielded numerous results in various inverse coefficients problems. For a comprehensive overview, one can refer to  \cite{harrach2012simultaneous,harrach2018localizing,harrach2010exact,harrach2017local,harrach2019global, eberle2020lipschitz} and the  references  therein.

In our paper, we showcase the applicability of this approach to establish novel Lipschitz stability results for the inverse problem of recovering the Lamè parameters  $\lambda,\mu$ and the density $\rho$, either simultaneously or separately.  Note that in the case of time harmonic regime  for the  wave equation a conditional Lipshitz stability result is derived  in \cite{beretta2017uniqueness}.
The authors  consider only the three dinesional case and assume that the paremeters  $\lambda,\mu$, $\rho$ are piecewise constants on a given partition.

Our results are distinguished by the fact that they require fewer restrictions, fitting any dimension $d\geq 2$ and partial  boundary data.
 These Lipschitz stability results, to the best of our knowledge, represent new contributions to the problem under consideration.

It is crucial to underscore that Lipschitz stability plays a vital role in the numerical reconstruction of unknown parameters, enabling the analysis of convergence rates in iterative algorithms \cite{de2012local,maarten2015analysis}. Additionally, it's important to note that, in practical scenarios, we only have access to a finite-dimensional approximation of the Neumann-to-Dirichlet operator. In light of the preprint by Alberti and Santacesaria \cite{alberti2019infinite}, our results on Lipschitz stability could prove beneficial in addressing the finite-dimensional approximation of the Neumann-to-Dirichlet operator.

The paper unfolds as follows. In Section 2, we introduce the forward model, the Neumann-to-Dirichlet operator, and present the inverse problem. The main theoretical tools for this work are provided in Sections 3 and 4.
Section 3 focuses on reconstructing the density profile $\rho$ under the assumption that the  Lamé parameters  $\lambda,\mu$ are known. We establish a monotonicity relation between $\rho$ and the Neumann-to-Dirichlet operator $\Lambda_\rho$. The section also includes the proof of the existence of localized potentials, leading to the derivation of non constructive and constructive Lipschitz stability results.
In Section 4, we address the inverse problem of simultaneously recovering the Lamé parameters $\lambda,\mu$ and the density $\rho$ from the Neumann-to-Dirichlet map $\Lambda_{\lambda,\mu,\rho}$. We establish a monotonicity relation between the coefficients $(\lambda,\mu,\rho)$ and the Neumann-to-Dirichlet operator $\Lambda_{\lambda,\mu,\rho}$. The section further includes the proof of the existence of localized potentials and the Lipschitz stability result.
\section{Problem formulation}
Let $\Omega \subset \mathbb{R}^d$ ($d\geq 2$),  be a bounded  domain with  $C^{1,1}$  boundary $\partial \Omega=\overline{\Gamma_N\cup\Gamma_D}$,
with  $\Gamma_N\cap\Gamma_D  = \emptyset$. For the following, we define 
\[
L^\infty_+(\Omega):=\{v\in L^\infty(\Omega): \text{ess\,inf}_{x\in\Omega}v(x)>0\}.
\] 
The divergence of a matrix $M\in \R^{d\times d}$ is defined by
$\nabla\cdot M=\sum_{i,j}^d\frac{\partial M_{i,j} }{\partial x_j}e_i$, where $e_i$  is a unit vector and $x_j$ is a component of a vector from $\R^d$. The Frobenius product of two matrix  $M, M^{\prime}\in \R^{d\times d}$ is defined as $M:M^{\prime}=\sum_{i,j}^dM_{i,j}M^{\prime}_{i,j}$.

Let $u:\Omega \rightarrow \R^d$ the dispslacement vector. For 
 $\lambda,\mu, \rho \in L^\infty_+(\Omega)$, we consider the following  problem  with Neumann boundary data  $g\in L^{2}(\Gamma_N,\mathbb{R}^d)$:
\begin{equation}
\label{direct}
\left\{
\begin{aligned}
 -\nabla\cdot\left(\lambda (\nabla\cdot u) I+2\mu\nabla^s u\right) +\rho u &= 0\quad \textrm{ in }  \Omega,\\
(\lambda(\nabla\cdot u) I+2\mu \nabla^s u)\nu &  = g\quad \textrm{ on  }\Gamma_N,\\
  u&=0 \quad  \textrm{ on  }  \Gamma_D,
 \end{aligned}\right.
\end{equation}
where $\nu$ denotes the unit outer normal   vector  at the boundary $\partial\Omega$,
$I$ is $d\times d$-identity matrix  and $\nabla^s u$ is a symetric matrix defined by 
$\nabla^s u=( \nabla u+(\nabla u)^T)/2$. 

 The  weak formulation of  problem (\ref{direct})  reads
\begin{equation}
\label{eqv1}
\int_\Omega \lambda \nabla\cdot u \nabla\cdot w+ 2\mu \nabla^s u:\nabla^s w\,dx +\int_\Omega \rho u\cdot w\,dx=\int_{\Gamma_N}g\cdot w\,ds \quad  \textrm{ for all } w\in \mathcal{V},
\end{equation}
where
\[
\mathcal{V}:=  \left\{ v\in  H^1(\Omega,\mathbb{R}^d):  \quad v|_{\Gamma_D}= 0  \right\}.
\]
\begin{hypo}
Assume that  $\mu(x)\in C^{0,1}(\Omega)$
and $\lambda(x), \rho(x)\in L^{\infty}_+(\Omega)$ satisfy
\[
\begin{aligned}
&\mu(x)\geq \delta_0, \quad \lambda(x)+2 \mu(x)\geq \delta_0\quad \mathrm{a.\,e.}\,\,x\in{ \Omega},\\
&\Vert \mu\Vert_{C^{0,1}({\Omega})}+\Vert\lambda\Vert_{L^{\infty}(\Omega)}\leq M_0, \quad  \Vert\rho\Vert_{L^{\infty}(\Omega)} \leq M_0,
\end{aligned}
\]
\noindent
with positive constants $\delta_0$, $M_0$, where we define $\Vert f\Vert_{C^{0,1}(\Omega)}=\Vert f\Vert_{L^{\infty}({ \Omega})}+\Vert\nabla f\Vert_{L^{\infty}(\Omega)}$.
\end{hypo}
  For $\lambda,\mu, \rho$ satisfaying   Hypothesis 1 and $g\in L^2(\Gamma_N,\mathbb{R}^d)$,  the existence and uniqueness of a solution  to the weak formulaton (\ref{eqv1}) can be shown by the Lax-Milgram theorem.
 Note that the parameters $\lambda,\mu, \rho$ must satisfy Hypothesis 1  in order to use  the unique continuation result from \cite{lin2010quantitative}.

Measuring boundary displacements $u_{\vert_{\Gamma_N}}$ that result from applying forces  $g$ to $\Gamma_N$ can be modeled by the  Neumann-to-Dirichlet  operator :
\[
 \begin{aligned}
 \Lambda_{\lambda,\mu, \rho} : L^{2}(\Gamma_N,\mathbb{R}^d) &\longrightarrow   L^{2}(\Gamma_N,\mathbb{R}^d) \\
                                & g \longmapsto u_{|\Gamma_N},
  \end{aligned}
 \]
 where $u$  is the unique solution   of \eqref{eqv1}.
 The  inverse problem we  consider here,  is the following:
\begin{equation}
\label{invp}
\textrm{ \it  Find the parameters } (\lambda,\mu,\rho) \textrm{  \it  from  the knowledge of the  map } 
 \Lambda_{\lambda,\mu, \rho}.  
 \end{equation}
 \section{Recovery of  the  density  $\rho$}
 In this section,  we assume  that the  Lamé coefficients   $\lambda, \mu$ are  known and we are intereted  to recover the density   $\rho$ from the Neumann-to-Dirichlet operator  $ \Lambda_{\lambda,\mu, \rho}$.  For simplicity,  we will denote    $ \Lambda_{\lambda,\mu, \rho}$ by   $ \Lambda_{\rho}$. 
 
The  Lipschitz stability  estimate is proven by  the monotonicity relation between the  density    $\rho$ and the Neumann-to-Dirichlet-operator $\Lambda_{\rho}$, and  the existence of localized potentials.  
 \subsection{Monotonicity, localized potentials  and Lipschitz stability}\label{sec3}
In the following subsection, we establish the existence of localized potentials and derive a Lipschitz stability estimate without assuming definiteness  on the parameter.
\subsubsection{Monotonicity}
In the following lemma, we show a monotonicity relation between the Neumann-to-Dirichlet operator and the density $\rho$. 
\begin{lemma}[Monotonicity relation]
\label{mono}
Let $\rho_1, \rho_2 \in L^\infty_+(\Omega)$  and let  $g\in L^2(\Gamma_{\textup N},\mathbb{R}^d)$  be an applied boundary load. The corresponding solutions of (\ref{direct}) are denoted by $u_1:=u^g_{\rho_1},\ u_2:=u^{g}_{\rho_2}\in \mathcal{V}$. Then we have
\begin{equation}
\label{eqmono}
\int_\Omega(\rho_1-\rho_2)\Vert u_2\Vert^2\,dx
\geq \langle g,\Lambda_{\rho_2}(g)\rangle-\langle g,\Lambda_{\rho_1}(g)\rangle
\geq \int_\Omega(\rho_1-\rho_2)\Vert u_1\Vert^2\,dx.
\end{equation}
\end{lemma}
\begin{proof}
Since $\Lambda_{\rho_2}(g)=u_2|_{\Gamma_N}$, we  use the variational  formulation \eqref{eqv1} for $\rho_1$ and $\rho_2$ with
$w:=u_2$ and obtain 
\[
\begin{aligned}
&\int_\Omega \lambda \nabla\cdot u_1\nabla\cdot u_2\,dx+
2\int_\Omega\mu\nabla^s u_1:\nabla^s u_2\,dx+ \int_\Omega\rho_1 u_1\cdot u_2\,dx\\
&=\langle g,\Lambda_{\rho_2}(g) \rangle\\
&=  \int_\Omega \lambda \nabla\cdot u_2\nabla\cdot u_2\,dx   +  \int_\Omega\mu\nabla^s u_2:\nabla^s u_2\,dx+ \int_\Omega\rho_2 \Vert u_2\Vert^2\,dx.
\end{aligned}
\]
Therefore 
\[
\begin{aligned}
&\int_\Omega \lambda \nabla\cdot (u_1-u_2)\nabla\cdot (u_1-u_2)\,dx+
2\int_\Omega\mu\nabla^s (u_1-u_2):\nabla^s(u_1-u_2)\,dx\\
&+ \int_\Omega\rho_1\Vert u_1-u_2\Vert^2\,dx \\
& =\int_\Omega \lambda \nabla\cdot u_1\nabla\cdot u_1\,dx+
\int_\Omega \lambda \nabla\cdot u_2\nabla\cdot u_2\,dx
-2\int_\Omega \lambda \nabla\cdot u_1\nabla\cdot u_2\,dx\\
&+2\int_\Omega\mu\nabla^s u_1:\nabla^s u_1\,dx
+2\int_\Omega\mu\nabla^s u_2:\nabla^s u_2\,dx
-4\int_\Omega\mu\nabla^s u_1:\nabla^s u_2\,dx\\
&+  \int_\Omega \rho_1\Vert u_1 \Vert^2\,dx +    \int_\Omega \rho_1\Vert u_2 \Vert^2\,dx -2  \int_\Omega \rho_1 u_1\cdot u_2\,dx\\
&=\langle g,\Lambda_{\rho_1}(g)\rangle-\langle g,\Lambda_{\rho_2}(g)\rangle
+\int_\Omega(\rho_1-\rho_2)\Vert u_2\Vert^2\,dx.
\end{aligned}
\]
We observe that  the left-hand side  of the above equality is nonnegative, thus  the first  inequality  in \eqref{eqmono} follows. 
To obtain the second inequality, we simply interchange the roles of
 $\rho_1$ and $\rho_2$.
\end{proof}
\noindent
 According to  the previous lemma, we have  the following monotonicity property.
\begin{corollary}[Monotonicity]\label{monotonicity}
For $\rho_1, \rho_2 \in L^\infty_+(\Omega)$, we have
\begin{equation}\label{monotonicity_corol}
\ \rho_1\leq \rho_2 \quad \text{  implies } \quad
\langle \Lambda_{\rho_1}(g), g\rangle\geq \langle \Lambda_{\rho_2}(g),g\rangle,
\quad \text{ for all } g\in L^2(\Gamma_N,\R^d).
\end{equation}
\end{corollary}
\subsubsection{Runge approximation and localized potentials}
In this subsection, we present a Runge approximation result that will be instrumental in establishing the existence of localized potentials.
\begin{lemma}[Runge approximation]
\label{thm:runge}
Let  $\lambda, \mu, \rho$  satisfying Hypothesis 1 and 
$D$ be subset of $\Omega$ with positive measure.  Denote $S_D$ the subset of $L^2(D)$ defined by 
\begin{equation}\label{set_sol}
S_D:=\{v\in L^2(D,\R^d):  -\nabla\cdot(\lambda \nabla\cdot v+2\mu \nabla^s v) +\rho v    =0 \text{ in } D\}.
\end{equation}
Then,  for all $f\in S_D$,  there exists a sequence $(g_n)_{n\in\mathbb N}\subset L^2(\Gamma_N,\R^d)$ such that the corresponding solutions $u^{(g_n)}$ of problem \ref{direct}   with  a boundary data $g_n$, $n\in \N$,  
fulfill
\[
u^{(g_n)}|_{D}\to f \quad \textrm{ in } L^2(D,\R^d).
\]
\end{lemma}
\begin{proof}
Denoting by  $X$ the closure of $S_D$ in $L^2(D,\R^d)$, wich is a  Hilbert space.  We introduce the operator
\begin{equation}\label{operA}
A:   X\to  L^2(\Gamma_N,\R^d), \quad f\mapsto A f:= v_{|\Gamma_N}, 
\end{equation}
where $v\in H^1(\Omega, \R^d)$  solves 
\begin{equation}\label{eq:runge_v}
\int_\Omega \lambda\nabla\cdot v\nabla\cdot w + 2\mu \nabla^s v: \nabla^s w\,dx+ \int_\Omega\rho  v\cdot w\,dx=\int_{D}f \cdot w\,ds, \quad \textrm{ for all } w\in 
\mathcal V.
\end{equation}
Let  $g\in L^2(\Gamma_N,\R^d)$ and $u\in H^1(\Omega, \R^d)$ be the corresponding solution  of the problem \eqref{direct} with Neumann boundary data $g$.  Then the adjoint operator of  $A$ is characterized by 
\begin{equation}
\begin{aligned}
 \int_{D} \left( A^* g \right) \cdot f \,dx& = \int_{\Gamma_N} \left( A f \right)\cdot g \,ds = \int_{\Gamma_N} v\cdot g \,ds \\
&=\int_\Omega \lambda\nabla\cdot u\nabla\cdot v + 2\mu \nabla^s u: \nabla^s v\,dx+ \int_\Omega\rho  u\cdot v\,dx\\
&=\int_{D}f\cdot u\,dx, \quad \textrm{ for all } f\in  X,
\end{aligned}
\end{equation}
which shows that $A^*: L^2(\Gamma_N,\R^d)\to  X$ fulfills 
$A^* g=u|_{D}$,  where $u$  is a solution  to  problem \eqref{direct}  with Neumann boundary data $g$.
 The assertion follows if we can show that $A^*$ has dense range, which is equivalent
to $A$ being injective.

To prove this, let $v|_{\Gamma_N}=Af=0$ with $v\in H^1(\Omega,\R^d)$ solving (\ref{eq:runge_v}).
Since (\ref{eq:runge_v}) also implies that $(\lambda(\nabla\cdot v) I+2\mu \nabla^s v)\nu|_{\Gamma_N}=0$, and $\Omega\setminus D$ is connected, it follows by unique continuation that $v|_{\Omega\setminus D}=0$ and thus $v^+|_{\partial D}=0$. Since $v\in H^1(\Omega,\R^d)$ this also implies that
$v^-|_{\partial D}=0$, and together with (\ref{eq:runge_v}) we obtain that $v|_{D}\in H^1(D) $ solves
\[
-\nabla\cdot(\lambda \nabla\cdot v+2\mu \nabla^s v) +\rho v =f \quad \textrm{ in } D,
\]
with homogeneous Dirichlet boundary data $v|_{\partial D}=0$. 
Let $w\in S_D$,  from (\ref{eq:runge_v}) it then follows that $\int_{D}f \cdot w\,dx=0$ for all $w\in  S_D$. This implies that $f$ is $L^2$-orthonal to $S_D$ and the proof is completed.
\end{proof}
In the following theorem, we demonstrate the controllability of the energy terms $\Vert u^g_\rho\Vert^2$ within the monotonicity relation \eqref{eqmono} by employing the method of localized potentials. Specifically, the energy $\Vert u^g_\rho\Vert^2$ can be intentionally increased to arbitrary levels within a specified subset $\mathcal{D}_1 \subset \Omega$, without causing a corresponding increase in another subset 
$\mathcal{D}_2 \subset \Omega$. This holds true, especially when access from the boundary $\partial \Omega$ to $\mathcal{D}_1$ is attainable without passing through $\mathcal{D}_2$.
\begin{theorem}[Localized potentials]
\label{locpot}
Let  $\lambda, \mu, \rho$  satisfying Hypothesis 1  and   $\mathcal{D}_1, \mathcal{D}_2$ be two  open sets with  $\overline{\mathcal{D}}_1, \overline{\mathcal{D}_2}\subset \Omega$, 
  $\overline{\mathcal{D}}_1\cap \overline{\mathcal{D}_2}=\emptyset$,  $\Omega\setminus\left(\overline{\mathcal{D}_1}\cup\overline{\mathcal{D}_2}\right)$ is connected
and  $\overline{\Omega\setminus\left(\overline{\mathcal{D}_1}\cup\overline{\mathcal{D}_2}\right)}\cap \Gamma_N \neq\emptyset$.  
Then,  there exists a sequence
$(g_n)_{n\in \mathbb{N}}\subset L^2(\Gamma_{\textup N}, \mathbb{R}^d)$, such that the corresponding solutions
$(u^{g_n})_{n\in \mathbb{N}}$ of \eqref{direct} fulfill 
\begin{align}
&\lim_{n\to \infty}\int_{\mathcal{D}_1}  \Vert  u^{ g_n}\Vert^2\,dx=\infty,\label{localized_grad_1}\\
&\lim_{n\to \infty}\int_{\mathcal{D}_2}  \Vert u^{ g_n} \Vert^2 \,dx=0.\label{localized_grad_2}
\end{align}
\end{theorem}
\begin{proof}
Let $\mathcal D_1$,  $\mathcal D_2$ defined as in Theorem \ref{locpot} and consider 
the function $\Phi$ solution of 
\[
-\nabla\cdot(\lambda \nabla\cdot \Phi+2\mu \nabla^s \Phi) +\rho \Phi =0 \quad \textrm{ in } \mathcal D_1, \quad (\lambda \nabla\cdot \Phi+2\mu \nabla^s \Phi)\nu= \tilde g \text{ on }
\partial \mathcal D_1
\]
\[
-\nabla\cdot(\lambda \nabla\cdot \Phi+2\mu \nabla^s \Phi) +\rho \Phi=0 \quad \textrm{ in } \mathcal D_2, \quad (\lambda \nabla\cdot \Phi+2\mu \nabla^s \Phi)\nu=  0 \text{ on }
\partial \mathcal D_2,
\]
where $0\neq \tilde g\in L^2(\partial \mathcal D_1,\R^d)$.   Obviously $\Phi \in S_D$ where $D=\mathcal D_1\cup\mathcal D_2$
and $\Phi\vert_{\mathcal D_1}\not\equiv 0$,   $\Phi\vert_{\mathcal D_2}\equiv 0$. Using the runge approximation result
in Lemma \ref{thm:runge},  we can find a sequence $(\tilde g_n)\subset L^2(\Gamma_N,\R^d)$ such  that  the corresponding 
solutions $u^{\tilde g_n}$ of problem \eqref{direct} with Neumann boundary data $ \tilde g_n$ fulfill
\[
u^{\tilde g_n}\rightarrow \Phi \text{ in }  L^2(D, \R^d).
\]
Consier the scaled  sequence 
\[
g_n =\frac{\tilde g_n}{\left(\int_{\mathcal D_2} \Vert u^{\tilde g_n} \Vert^2\,dx\right)^{1/4}},
\]
we obtain 
\begin{align*}
&\int_{\mathcal{D}_1}  \Vert  u^{ g_n}\Vert^2\,dx=  \frac{\int_{\mathcal{D}_1}  \Vert  u^{\tilde g_n}\Vert^2\,dx}{\sqrt{\int_{\mathcal{D}_2}  \Vert  u^{\tilde g_n}\Vert^2\,dx}}
\rightarrow \infty \quad\text{ as  }  n\rightarrow \infty, \\
&\int_{\mathcal{D}_2}  \Vert u^{ g_n} \Vert^2 \,dx=  \frac{\int_{\mathcal{D}_2}  \Vert  u^{\tilde g_n}\Vert^2\,dx}{\sqrt{\int_{\mathcal{D}_2}  \Vert  u^{\tilde g_n}\Vert^2\,dx}}
\rightarrow 0 \quad \text{ as  }  n\rightarrow \infty,
\end{align*}
and the proof is completed.
\end{proof}
\subsubsection{Lipschitz stability result}
In the following theorem, we state the main  result of this section.
Let  $\mathcal{E}$ be a finite dimensional subspace of $L^\infty(\Omega)$.  We consider two constants \mbox{$0<a\leq b$},  which are the lower and upper bounds of the  density 
$\rho$  and define the set
\[
\mathcal{E}_{[a,b]}=\left\{ \rho\in \mathcal{E}:  \quad  a\leq\rho(x)\leq b,   \quad \text{ for all } x\in \Omega \right\}.
\]
The domain $\Omega$, the finite-dimensional subspace $\mathcal{E}
$ and the bounds $0< a\leq b$   are fixed, and the constant in the Lipschitz stability result
 will depends on them.
\begin{theorem}[Lipschitz stability]
\label{stability}
 There exists a positive constant $C>0$ such that for all $\rho_1, \rho_2 \in \mathcal{E}_{[a,b]}$,  we have 
\begin{equation}
\label{stab-est}
 \Vert\rho_1-\rho_2\Vert_{L^\infty(\Omega)}
\leq   C\| \Lambda(\rho_1)-\Lambda(\rho_2) \|_*.
\end{equation}
Here  $\Vert.\Vert_*$ is the operator norm of $\mathcal L(L^2(\Gamma_{\textup N},\R^d))$. 
\end{theorem}
To prove Theorem \ref{stability},  we  first prove  the   following lemmas.
\begin{lemma}
\label{lem_lower}
For $\rho_1, \rho_2 \in \mathcal{E}_{[a,b]}$  with   $\rho_1 \neq \rho_2$,  we have
\[
\frac{\| \Lambda_{\rho_1}-\Lambda_{\rho_2} \Vert_*}{\Vert \rho_1-\rho_2\Vert_{L^\infty(\Omega)}}\geq
\inf_{\substack{\zeta\in \mathcal{K},\\ \tau_1,\tau_2\in\mathcal{E}_{[a,b]}}} \sup_{\| g \|=1} \max\left\{ J\left(g,\zeta,\tau_1\right), J\left(g,-\zeta,\tau_2\right) \right\},
\]
where 
\begin{equation}\label{func_J}
J :\ L^2(\Gamma_{\textup N},\R^d)\times \mathcal{E}\times \mathcal{E}_{[a,b]}\to \mathbb{R}: 
\left(g,\zeta,\tau\right)\mapsto \int_{\Omega} \zeta  u_{\tau}^{g} \cdot  u_{\tau}^{g}\,dx,
\end{equation}
and 
\begin{equation}\label{set_K}
 \mathcal{K}:=\left\{\zeta\in \mathcal{E}:  \quad 
\Vert \zeta\Vert_{L^\infty(\Omega)}= 1 \right\}.
\end{equation}
\end{lemma}
\begin{proof}
For the sake of brevity, we  will write  $\Vert g \Vert$  for  $\Vert g\Vert_{L^2(\Gamma_{\textup N},\R^d)}$.
The operators  $\Lambda_{\rho_1}$ and $\Lambda_{\rho_2}$ are self-adjoint,
so  that 
\begin{align*}
\Vert\Lambda_{\rho_1}-\Lambda_{\rho_2}\Vert_*&= 
  \sup_{\Vert g\Vert=1} \Vert \langle g, \left(\Lambda_{\rho_1}-\Lambda_{\rho_2}\right) g\rangle\Vert\\
&= \sup_{\Vert g\Vert=1}\max\left\{\langle g, \left(\Lambda_{\rho_1}-\Lambda_{\rho_2}\right) g\rangle,  \langle g, \left(\Lambda_{\rho_2}-\Lambda_{\rho_1}\right) g\rangle  \right\}.
\end{align*}
Next, we apply the monotonicity relation \eqref{eqmono} in Lemma \ref{mono} and  we  obtain 
for all $g\in L^2(\Gamma_{\textup N},\R^d)$ that
 \begin{align}\label{estim_1}
\Vert\Lambda_{\rho_1}-\Lambda_{\rho_2}\Vert_*
\geq \sup_{\Vert g\Vert=1}\max\left\{ \int_\Omega(\rho_1-\rho_2)\Vert  u_{\rho_1}^{g}\Vert^2\,dx,
-\int_\Omega(\rho_1-\rho_2)\Vert u_{\rho_2}^{g}\Vert^2 \,dx  \right\},
 \end{align}
where $u_{\rho_j}^{g}\in \mathcal{V}$ denotes the solution of \eqref{direct} with Neumann data $g$ and  the density  $\rho_j$.

The estimate (\ref{estim_1}) contains the linear difference  $\rho_1-\rho_2$,
but it also contains the solutions $u_{\rho_1}^{g}$  and $u_{\rho_2}^{g}$  that depend non-linearly on the coefficients.
Following the ideas of \cite{harrach2019global,harrach2019uniqueness}, we obtain for $\rho_1\neq \rho_2$,
\begin{align}\label{estim_3}
\nonumber
&\frac{\Vert \Lambda_{\rho_2}-\Lambda_{\rho_1} \Vert_*}{\Vert\rho_1-\rho_2\Vert_{L^\infty(\Omega)}} \\ 
&\geq   \sup_{\Vert g\Vert=1}\max\left\{  J\left(g,\frac{\rho_1-\rho_2}{\Vert \rho_1-\rho_2 \Vert_{L^\infty(\Omega)}}, \rho_1\right),
  J\left(g,\frac{\rho_2-\rho_1}{\Vert \rho_1-\rho_2 \Vert_{L^\infty(\Omega)}}, \rho_2\right)\right\}.
\end{align}
The second argument of the function  $h$ in (\ref{estim_3}) belongs to the compact set
\begin{align*}
 \mathcal{K}&:=\left\{\zeta\in \mathcal{E}:  \quad 
\Vert \zeta\Vert_{L^\infty(\Omega)}= 1 \right\}.
\end{align*}
Hence, (\ref{estim_3}) yields that
\begin{align*}
\frac{\| \Lambda_{\rho_1}-\Lambda_{\rho_2} \Vert_*}{\Vert \rho_1-\rho_2\Vert_{L^\infty(\Omega)}}\geq
\inf_{\substack{\zeta\in \mathcal{K},\\ \tau_1,\tau_2\in\mathcal{E}_{[a,b]}}} \sup_{\| g \|=1} \max\left\{ J\left(g,\zeta,\tau_1\right), J\left(g,-\zeta,\tau_2\right) \right\},
\end{align*}
 and the proof is completed.
 \end{proof}
 In the next lemma, we use a compactness argument to show that the lower  bound  in Lemma \ref{lem_lower} attains  its  infimum. 
 \begin{lemma}
 \label{lem_inf}
 There exists  $(\hat \zeta,  \hat \tau_1,  \hat \tau_2)\in \mathcal{K}\times \mathcal{E}_{[a,b]}\times \mathcal{E}_{[a,b]} $ such that
 \[
 \inf_{\substack{\zeta\in \mathcal{K},\\ \tau_1,\tau_2\in\mathcal{E}_{[a,b]}}} \sup_{\| g \|=1} \max\left\{ J\left(g,\zeta,\tau_1\right), J\left(g,-\zeta,\tau_2\right) \right\}
 =  \sup_{\| g \|=1} \max\left\{ J\left(g,\hat \zeta,\hat\tau_1\right), J\left(g,-\hat\zeta,\hat\tau_2\right) \right\}.
 \]
 \end{lemma}
 \begin{proof}
 Since $(g,\zeta,\tau_1,\tau_2)\rightarrow\max\left\{ J\left(g,\zeta,\tau_1\right), J\left(g,-\zeta,\tau_2\right) \right\}$ is continuous, the function
\[
(\zeta,\tau_1,\tau_2)\rightarrow \sup_{\| g \|=1}\max\left\{ J\left(g,\zeta,\tau_1\right), J\left(g,-\zeta,\tau_2\right) \right\}
\]
is lower semi-continuous, so that it attains its infimum on  the compact set  $\mathcal{K}\times \mathcal{E}_{[a,b]}\times\mathcal{E}_{[a,b]}$.
\end{proof}
  \begin{proof}[Proof of theorem \ref{stability}]
To prove  Theorem \ref{stability},  it remains to show  that the infimum attained  in Lemma \ref{lem_inf} is strictly positive.  Hence,  it suffices to show that
\begin{align}\label{estim_4}
 \sup_{\Vert g \Vert=1}\max\left\{ J\left(g,\zeta,\tau_1\right), J\left(g,-\zeta,\tau_2\right) \right\}>0
\quad \text{ for all }  \zeta\in \mathcal{K}, \tau_1, \tau_2\in \mathcal{E}_{[a,b]}.
\end{align}
In order to prove that (\ref{estim_4}) holds true, let $\left(\zeta,\tau_1,\tau_2\right)\in \mathcal{K}\times \mathcal{E}_{[a,b]}\times \mathcal{E}_{[a,b]}$.
Then there exist an open subset $\emptyset\neq \mathcal{D}_1\subset \Omega$ and  a constant $0<\beta<1$,  such that 
\[
\begin{aligned}
&\text{(i)}\   \zeta|_{ \mathcal{D}_1}\geq \beta , \text{ or }\\
&\text{(ii)}\   -\zeta|_{ \mathcal{D}_1}\geq \beta.
\end{aligned}
\]
We use the localized potentials sequence in Theorem~\ref{locpot} to obtain an open subset $ \mathcal{D}_2\subset \Omega$  with  $\overline{ \mathcal{D}_1}\cap \overline{ \mathcal{D}_2}=\emptyset$, and a  boundary load
$\tilde g\in L^2(\Gamma_{\textup N}, \R^d)$ with 
\begin{align}\label{estim_loc_pot}
\int_{ \mathcal{D}_1} \Vert u^{\tilde g}_{\tau_1} \Vert^2\,dx   \geq \frac{2}{\beta} \quad
 \text{ and } \quad \int_{ \mathcal{D}_2} \Vert u^{\tilde g}_{\tau_1} \Vert^2\,dx  \leq 1.
\end{align}

In case $(i)$, this leads to
\[
\begin{aligned}
&J\left(\tilde g,\zeta,\tau_1\right)= \int_\Omega \zeta \Vert  u^{\tilde g}_{\tau_1}\Vert^2\, dx 
 \geq \int_{\mathcal{D}_1} \zeta \Vert  u^{\tilde g}_{\tau_1}\Vert^2\, dx
 + \int_{\mathcal{D}_2} \zeta \Vert u^{\tilde g}_{\tau_1}\Vert^2\, dx\\
 & \geq  \beta\int_{\mathcal{D}_1}  \Vert u^{\tilde g}_{\tau_1}\Vert^2\, dx
 -\int_{\mathcal{D}_2}  \Vert  u^{\tilde g}_{\tau_1}\Vert^2\, dx
 \geq 2- 1=1,
 \end{aligned}
\]
and in case (ii),   we can analogously use a localized potentials sequence for $\tau_2$ and find   $\hat g\in L^2(\Gamma_{\textup N}, \R^d)$  with
\[
\begin{aligned}
&J\left(\hat g,-\zeta,\tau_2\right)=\int_\Omega (-\zeta) \Vert u^{\tilde g}_{\tau_2}\Vert^2\, dx 
\geq  \beta  \int_{\mathcal{D}_1}\Vert  u^{\hat g}_{\tau_2}\Vert^2\, dx- \int_{\mathcal{D}_2}\Vert u^{\hat g}_{\tau_2}\Vert^2\, dx   \geq 1.
 \end{aligned}
\]
Hence, in both  cases, 
\[
\begin{aligned}
 \sup_{\| g \|=1}\max\left\{ J\left(g,\zeta,\tau_1\right), J\left(g,-\zeta,\tau_2\right) \right\}
&\geq  \max\left\{ J\left(\frac{\tilde g}{\Vert \tilde g\Vert},\zeta,\tau_1\right), J\left(\frac{\hat g}{\Vert \hat g\Vert},-\zeta,\tau_2\right) \right\}\\
&=  \max\left\{ \frac{1}{\Vert\tilde g\Vert^2} J\left({\tilde g},\zeta,\tau\right), \frac{1}{\Vert\hat g\Vert^2}J\left({\hat g},-\zeta,\tau\right) \right\}>0,  
\end{aligned}
\]
so that Theorem \ref{stability} is proven.
\end{proof}
\begin{remark} 
As a consequence of Theorem \ref{stability}, we end up with the following uniqueness result.  For 
all  $\rho_1, \rho_2\in \mathcal{E}_{[a,b]}$, we have 
\[
\Lambda_{\rho_1}=\Lambda_{\rho_2}\quad \text{ if and only if }\quad \rho_1=\rho_2.
\]
\end{remark}
\subsubsection{Constructive Lipschitz stability estimate} \label{quant}
In this subsection, we assume that $\text{supp}(\rho)=\mathcal S\Subset \Omega$  and we  restrict ourself   to  the case where  $\mathcal{E}$  is a  set of piecewise constant  functions  
of  the form 
\[
\mathcal{E}=\left\{\rho(x)=\sum_{j=1}^N \rho_j\chi_{S_j},\quad  \rho_1,\ldots \rho_N\in\R^*_+\right\}\subset L^\infty(\mathcal S),
\]
where  $S_j, j=1,\ldots,N$  are known open connected    sets and pairwise nonoverlapping such that  $\cup_{j=1}^N S_j=\mathcal S$
and   $\partial S_j, j=1,\ldots,N$   are piecewise smooth and touch the boundary $\partial\mathcal S$.
\begin{remark}
The condition that  $S_j$    intersects  the boundary   $\partial\mathcal S$ is required  to prove Lemma \ref{existence}, which is the key to prove the constructive Lipschitz stability.
\end{remark}
For  $0<a< b$,   $\mathcal{E}_{[a,b]}$
is the set  of $\rho\in \mathcal{E}$  such  that $a\leq \rho_j\leq  b$  for all $j=1,\ldots, N$.
The structure   assumed   for $\rho$  fits well in several  problems  arising in practical applications.\\
\noindent
To establish our  constructive Lipschitz stability estimate, a finite set of localized potentials is required. We illustrate the process of reconstructing these potentials.
\begin{lemma}\label{existence}
Let $b>a>0$ be given constants. For $j=1,\ldots, N$ and $k=1,\ldots,K$,  with $K=\left(\lfloor 5\left( \frac{b}{a}-1 \right)\rfloor +1 \right)$, where $\lfloor.\rfloor$ is the floor function. We  define the piecewise constant function $\zeta^{(j,k)}\in L_{+}^\infty(\mathcal S)$  by
\[
\zeta^{(j,k)}(x)=
(k+6)\frac{a}{5} \mathbbm{1}_{S_j}(x) +
\frac{3a}{5} \mathbbm{1} _{\mathcal S\setminus S_j}(x).
\]
\begin{itemize}
\item[(i)] There exist  boundary data  $g^{(j,k)}\in  L^2(\Gamma_N,\R^d)$,  so  that the corresponding  solutions
$u^{g^{(j,k)}}_{\zeta^{(j,k)}}\in \mathcal V$ of  (\ref{direct})  with  $g= g^{(j,k)}$,  $(\lambda, \mu)$ satisfying Hypotheis 1  and $\rho=\zeta^{(j,k)}$  fulfill
\begin{equation}
\label{exist_g}
 I^{(j,k)}:= \frac{1}{2}\int_{S_j} \vert u^{g^{(j,k)}}_{\zeta^{(j,k)}}\vert^2\,dx-\left( \frac{5b}{2a} -\frac{3}{2} \right)\int_{\mathcal S\setminus S_j}
\vert u^{g^{(j,k)}}_{\zeta^{(j,k)}}\vert^2\,dx  > 0.
\end{equation}
\item[(ii)]
For arbitrary  $\rho\in \mathcal{E}_{[a,b]}$,  the solutions    $u^{g^{(j,k)}}_\rho\in \mathcal V$ of  (\ref{direct})  with  $g= g^{(j,k)}$,  
 $(\lambda, \mu)$ satisfying Hypotheis 1  and $\rho=\zeta^{(j,k)}$
fulfill
\[
 \int_{S_j}\vert u^{g^{(j,k)}}_\rho \vert^2\,dx  - \int_{\mathcal S\setminus S_j}\vert u^{g^{(j,k)}}_\rho \vert^2\,dx \geq  I^{(j,k)} >
0.
\]
\item[(iii)]   $g^{(j,k)}$ can be computed  by solving a finite number of well-posed PDEs.
\end{itemize}
\end{lemma}
\begin{proof}
 The proof of $(i)$ follows immediately  from the localized potentials  result in  Theorem \ref{locpot}. To prove $(ii)$,  we  need the following monotonicity  result which follows  from  Lemma \ref{mono} with $\rho_2=\rho+\delta$  and $\rho_1=\rho$
 and from using the same inequality again with interchanged roles of $\rho_1$  and $\rho_2$.
 For $g\in L^2(\Gamma_N,\R^d)$, $\rho\in L^{\infty}_+(\mathcal S)$,    and $\delta \in L^{\infty}_+(\mathcal S) $   such  that 
 $(\rho+\delta) \in L^{\infty}_+(\mathcal S)$,  we have
 \begin{equation}\label{monob}
 \int_{\mathcal S} \delta u^g_\rho\,dx \geq    \int_{\mathcal S} \delta u^g_{\rho+\delta}\,dx.
 \end{equation}
 Let $j=1,\ldots,N$ and   $\rho\in \mathcal{E}_{[a,b]}$.  Since $K$  fullfils  $b< (K+5)\frac{a}{5}$,  there exists $k\in \{1,\ldots, K \}$
 such   that  $\rho_j=  \rho|_{S_j}$  satisfies 
 \[
(k+4) \frac{a}{5}\leq  \rho_j<  (k+5) \frac{a}{5}.
 \]
 Applying the monotonicity-based inequality (\ref{monob}),   with  
 \[
  \frac{a}{5}\leq  (k+6) \frac{a}{5}- \rho_j<  \frac{2a}{5} \quad  \text{ and } \quad -b+ \frac{3a}{5}\leq   \frac{3a}{5}- \rho_j<  -\frac{2a}{5},
  \]
  we obtain 
  \[
  \begin{aligned}
  &  \int_{S_j}\vert u^{g^{(j,k)}}_\rho \vert^2\,dx  - \int_{\mathcal S\setminus S_j}\vert u^{g^{(j,k)}}_\rho \vert^2\,dx \\
& = \frac{5}{2a}\left(\int_{S_j} \frac{2a}{5}\vert u^{g^{(j,k)}}_\rho \vert^2\,dx 
   -  \frac{2a}{5}\int_{\mathcal S\setminus S_j}\vert u^{g^{(j,k)}}_q \vert^2\,dx\right) \\
  & \geq \frac{5}{2a}\left(\int_{S_j} \left( (k+6 ) \frac{a}{5}- \rho_j\right) \vert u^{g^{(j,k)}}_\rho \vert^2\,dx  +\int_{\mathcal S\setminus S_j}\left(\frac{3a}{5}-\rho_j\right)\vert u^{g^{(j,k)}}_\rho\vert^2\,dx\right)\\
  &= \frac{5}{2a}\left(\int_{S_j} \left(\zeta^{(j,k)}- \rho_j\right) \vert u^{g^{(j,k)}}_\rho \vert^2\,dx  +\int_{\mathcal S\setminus S_j}\left(\zeta^{(j,k)}-\rho_j\right)\vert u^{g^{(j,k)}}_\rho \vert^2\,dx\right)\\
  & \geq \frac{5}{2a}\left(\int_{S_j} \left(\zeta^{(j,k)}- \rho_j\right) \vert u^{g^{(j,k)}}_{\zeta^{(j,k)}} \vert^2\,dx  +\int_{\mathcal S\setminus S_j}\left(\zeta^{(j,k)}-\rho_j\right)\vert u^{g^{(j,k)}}_{\eta^{(j,k)}} \vert^2\,dx\right)\\
  &\geq \frac{5}{2a}\left(\int_{S_j} \frac{a}{5} \vert u^{g^{(j,k)}}_{\zeta^{(j,k)}} \vert^2\,dx  -\int_{\mathcal S\setminus S_j}\left(b-\frac{3a}{5}\right)\vert u^{g^{(j,k)}}_{\zeta^{(j,k)}} \vert^2\,dx\right)\\
  &= \frac{1}{2} \int_{S_j}  \vert u^{g^{(j,k)}}_{\zeta^{(j,k)}} \vert^2\,dx-\left( \frac{5b}{2a} -\frac{3}{2} \right)
  \int_{\mathcal S\setminus S_j}\vert u^{g^{(j,k)}}_{\zeta^{(j,k)}} \vert^2\,dx=I^{(j,k)} >0,
  \end{aligned}
  \]
  and $(ii)$ is   proved.  To prove  $(iii)$,  we use a similar approach  as in the construction of localized potentials in   \cite{harrach2019global}. 
Denote $S_{D}$ the set defined in \eqref{set_sol}  with  $D\subset\mathcal S$ and $X$ the closer of  $S_{D}$ in $L^2(\mathcal S)$.
For $j=1,\ldots, N$  and  $k=1,\ldots,K$,   we introduce the operator $H$  as follows: 
 \[
H:  X\to  L^2(\Gamma_N,\R^d), \quad f\mapsto  H f:= v_{|\Gamma_N}, 
\]
where $v\in \mathcal V$  solves 
\begin{equation}   \label{eqvv}
\int_\Omega \lambda\nabla\cdot v\nabla\cdot w + 2\mu \nabla^s v: \nabla^s w\,dx+ \int_{\mathcal S}\zeta^{(j,k)}  v\cdot w\,dx=\int_{D}f \cdot w\,ds, \quad \textrm{ for all } w\in 
\mathcal V.
\end{equation}
We can easily shown  shown  that the adjoint operator  $H^*$ of   $H$  is given  by 
\[
H^*:  L^2(\Gamma_N,\R^d)\rightarrow   X:  g\mapsto  u\vert_{D},
\]
 where $u$  is the solution of   (\ref{direct})  with  $\rho=\zeta^{(j,k)}$,  and that $H^*$  has dense range.\\
 \noindent
As  in the proof of Theorem \ref{locpot},  for $B\subset S_j$ with positive measure and 
$D =B\cup (\mathcal S\setminus S_j)$, we can construct a function  $\Phi$ such  that 
$\Phi\vert_B\not\equiv 0$,   $\Phi\vert_{\mathcal S\setminus S_j}\equiv 0$ and 
$\Phi\in S_D$.

 Consider the linear  ill-posed equation
 \begin{equation}\label{nonleq}
 H^*g=\Phi.
 \end{equation}
 Sine $ \Phi \in \overline{\mathcal{R}(H^*)}$,  the CGNE method, yields  a sequence  of iterates 
 $(g_n)_{n\in \mathbb{N}}\subset L^2(\Gamma_N,\R^d)$      for  which 
 \[ 
 H^*g_n \rightarrow \Phi.
 \] 
Therefore,  the solutions $u_n$  of   (\ref{direct})    with  $\rho=\zeta^{(j,k)}$  and  $g=g_n$  fulfill
\[
 \frac{1}{2}\int_{S_j} \vert u_n\vert^2\,dx-\left( \frac{5b}{3a} -\frac{3}{2} \right)\int_{\mathcal{S} \setminus S_j}
\vert u_n\vert^2\,dx\rightarrow  \frac{1}{2}\int_{S_j} \vert \Phi\vert^2\,dx>0,
\]
so that after finitely many iteration steps, (\ref{exist_g}) is fulfilled.
\end{proof}
\noindent
Now, we state the main result of this  subsection.
 \begin{theorem}[Constructive Lipschitz stability estimate]\label{quantitative}
 Let $g^{(j,k)}\in L^2(\Gamma_N,\R^d)$  defined as in Lemma  \ref{existence}.   
Then 
 \begin{equation}\label{stab_q}
 \Vert  \rho_1-\rho_2  \Vert_{\infty}\leq \alpha\Vert \Lambda(\rho_1)-\Lambda(\rho_2) \Vert_*,  \quad \text{ for all }\quad
  \rho_1,  \rho_2\in \mathcal{E}_{[a,b]},
 \end{equation}
where
 \[
 \alpha= \left\{\max\left\{  \Vert  g^{(j,k)}\Vert^2_{L^2(\Gamma_N,\R^d)}, \quad j=1,\ldots, N,  k=1,\ldots, K  \right\}\right\}^{-1}.
 \]
 \end{theorem}
 \begin{proof}
 From  Lemma \ref{existence},  we have for all  $\rho\in \mathcal{E}_{[a,b]}$,  and for all $j\in \{1,\ldots,N\}$, 
 $k\in \{1,\ldots,K\}$
 \begin{equation}\label{eq_existence}
 \begin{aligned}
 &\sup_{\Vert g\Vert=1}\left( \int_{S_j}\vert u^g_\rho \vert^2\,dx - \int_{\mathcal S\setminus S_j}\vert u^g_\rho \vert^2\,dx\right)\\
&= \sup_{0\neq g\in L^2(\Gamma_N,\R^d)}\frac{1}{\Vert g\Vert^2}\left( \int_{S_j}\vert u^g_\rho \vert^2\,dx- \int_{\mathcal S\setminus S_j}\vert u^g_\rho \vert^2\,dx\right)\\
&\geq \frac{1}{\Vert g^{(j,k)}\Vert^2} \left( \int_{S_j}\vert u^{g^{(j,k)}}_\rho \vert^2\,dx- \int_{\mathcal S\setminus S_j}\vert u^{g^{(j,k)}}_\rho \vert^2\,dx\right)
\geq  \frac{1}{\Vert g^{(j,k)}\Vert^2} \geq  \alpha . 
\end{aligned}
 \end{equation}
 To prove Theorem \ref{quantitative},  it suffices  to show  that  
  \begin{equation}\label{eq:phi}
  \sup_{\| g \|=1}\max\left\{ J(g,\zeta,\tau_1), J(g,-\zeta,\tau_2)\right\}\geq \alpha \quad \text{ for all } (\zeta,\tau_1, \tau_2)\in \mathcal{K}\times\mathcal{E}_{[a,b]}\times \mathcal{E}_{[a,b]},
  \end{equation}
   where  $J$  and $\mathcal{K}$ defined in  (\ref{func_J}) and  (\ref{set_K}).  
 Since $\mathcal{E}$ contains only piecewise-constant functions,  for every $\zeta \in  \mathcal{K}$,  there
must exist a  subset  $S_j\subset \mathcal S$ with either
 \[
 \zeta\vert_{S_j}=1,  \quad \text{  or }\quad    \zeta\vert_{S_j}=-1.
 \]
 Hence using   (\ref{eq_existence}),  we obtain for the case $\zeta\vert_{S_j}=1$,
 \[
 \begin{aligned}
  \sup_{\| g \|=1} \max\left\{ J(g,\zeta,\tau_1), J(g,-\zeta,\tau_2)\right\}&\geq   \sup_{\| g \|=1}\int_\mathcal S\zeta \vert u^g_{\tau_1} \vert^2\,dx\\
& \geq  \sup_{\| g \|=1}\left(  \int_{S_j} \vert u^g_{\tau_1} \vert^2\,dx-
 \int_{\mathcal S\setminus S_j} \vert u^g_{\tau_1} \vert^2\,dx
 \right)\geq \alpha,
 \end{aligned}
 \]
 and for the case  $\zeta\vert_{S_j}=-1$,
  \[
  \begin{aligned}
  \sup_{\| g \|=1} \max\left\{ J(g,\zeta,\tau_1), J(g,-\zeta,\tau_2)\right\}&\geq   \sup_{\| g \|=1}\int_\mathcal S(-\zeta) \vert u^g_{\tau_2} \vert^2\,dx\\
& \geq  \sup_{\| g \|=1}\left(  \int_{S_j} \vert u^g_{\tau_2} \vert^2\,dx-
 \int_{\mathcal S\setminus S_j} \vert u^g_{\tau_2} \vert^2\,dx
 \right)\geq \alpha.
 \end{aligned}
 \]
 so that (\ref{eq:phi}) is proved and the proof is completed.
  \end{proof}
\section{Simultaneous recovery  of $\lambda, \mu$  and $\rho$}
The inverse problem of recovering  the Lamé coefficients   $\lambda, \mu$  and the density  $\rho$ simultaneously is known to be higly  ill-posed problem and
stability results can only be obtained under a-priori assumptions on  the ceofficients.
For our problem, we will prove a stability result under the assumption that 
the coefficients   have  a-priori known upper and lower bounds, belong to an a-priori known finite-dimensional subspace and that a definiteness condition holds.

The main tools employed to establish the Lipschitz stability estimate are once again the monotonicity relation between the coefficients and the Neumann-to-Dirichlet operator, along with the existence of localized potentials. These pivotal elements are explored further in the subsequent subsection.
\subsection{Monotonicity and  localized potentials  }\label{sec3}
Within this subsection, we demonstrate a monotonicity relationship between the coefficients $\lambda, \mu, \rho$ and the Neumann-to-Dirichlet operator$\Lambda_{\lambda, \mu, \rho}$. We establish the existence of localized potentials and then employ these findings to derive a Lipschitz stability estimate.
 \begin{lemma}[Monotonicity estimate]
\label{mono2}
Let $\lambda_1, \lambda_2, \mu_1, \mu_2, \rho_1, \rho_2  \in L^\infty_+(\Omega)$,  and let  $g\in L^2(\Gamma_{\textup N},\R^d)$ be an applied boundary load. The corresponding solutions of (\ref{direct}) are denoted by $u_1:=u^g_{\lambda_1,\mu_1,\rho_1},\ u_2:=u^{g}_{\lambda_2,\mu_2,\rho_2
      }\in \mathcal{V}$. Then
\begin{equation}
\label{mono1}
\begin{aligned}
&\int_\Omega(\lambda_1-\lambda_2)\Vert\nabla\cdot u_2 \Vert^2\,dx+2\int_\Omega(\mu_1-\mu_2)\Vert\nabla^s u_2\Vert^2_F\,dx + \int_\Omega (\rho_1-\rho_2)\Vert u_2\Vert^2\,dx\\
&\geq \langle g,\Lambda_{\lambda_2,\mu_2,\rho_2}g\rangle-\langle g,\Lambda_{\lambda_1,\mu_1,\rho_1}g\rangle \geq\\
&\int_\Omega(\lambda_1-\lambda_2)\Vert\nabla\cdot u_1 \Vert^2\,dx+ 2\int_\Omega(\mu_1-\mu_2)\Vert\nabla^s u_1\Vert^2_F\,dx +  \int_\Omega(\rho_1-\rho_2)\Vert u_1\Vert^2\,dx.
\end{aligned}
\end{equation}
\end{lemma}
\begin{proof}
Since $\Lambda_{\lambda_2\mu_2,\rho_2}g=u_2|_{\Gamma_N}$, we can use the variational formulation \eqref{eqv1} for $\lambda_1,\mu_1, \lambda_2,\mu_2, \rho_1$  and $\rho_2$ with
$v:=u_2$ and obtain
\[
\begin{aligned}
&\int_\Omega\lambda_1(\nabla\cdot u_1)(\nabla\cdot u_2)\,dx+
2\int_\Omega\mu_1\nabla^s u_1:\nabla^s u_2\,dx +\int_\Omega \rho u_1\cdot u_2\,dx\\
& =\langle g,\Lambda_{\lambda_2,\mu_2, \rho_2}g \rangle\\
&=\int_\Omega\lambda_2\Vert \nabla\cdot u_2\Vert^2\,dx+2\int_\Omega\mu_2\Vert\nabla^s u_2\Vert^2\,dx + \int_\Omega \rho_2 \Vert u_2\Vert^2\,dx.
\end{aligned}
\]
Thus 
\[
\begin{aligned}
\lefteqn{
\int_\Omega \lambda_1\Vert \nabla\cdot(u_1-u_2)\Vert^2\,dx+ 
2\int_\Omega\mu_1\Vert \nabla^s(u_1-u_2)\Vert^2_F\,dx +\int_\Omega \rho_1 \Vert u_1-u_2\Vert ^2\,dx}\\
& =\int_\Omega\lambda_1\Vert \nabla\cdot u_1\Vert^2\,dx+2\int_\Omega\mu_1\Vert\nabla^s u_1\Vert^2_F\,dx +\int_\Omega \rho_1 \Vert u_1 \Vert^2\,dx\\
&+\int_\Omega\lambda_1\Vert \nabla\cdot u_2\Vert^2\,dx+2\int_\Omega\mu_1\Vert \nabla^s u_2\Vert^2_F\,dx  + \int_\Omega \rho_1 \Vert u_2 \Vert^2\,dx\\
&-2\int_\Omega \lambda_1(\nabla\cdot u_1)(\nabla\cdot u_2)\,dx-4\int_\Omega\mu_1\nabla^s u_1:\nabla^s u_2\,dx-2\int_\Omega \rho_1 u_1\cdot u_2\,dx\\
&=\langle g,\Lambda_{\lambda_1\mu_1,\rho_1}g\rangle-\langle g,\Lambda_{\lambda_2,\mu_2,\rho_2}g\rangle\\
&+\int_\Omega(\lambda_1-\lambda_2)\Vert \nabla \cdot u_2 \Vert^2\,dx
+2\int_\Omega(\mu_1-\mu_2)\Vert \nabla^s u_2\Vert^2_F\,dx + \int_\Omega(\rho_1-\rho_2)\Vert u_2\Vert^2\,dx.
\end{aligned}
\]
Since the left-hand side is nonnegative, the first asserted inequality follows. 
\\
Interchanging $\lambda_1,\lambda_2,\mu_1, \mu_2$  and $\rho_1, \rho_2$,  the second inequality follows.
\end{proof}
\noindent
 Based on the previous lemma, we  have  the following monotonicity property.
\begin{corollary}[Monotonicity]\label{monotonicity}
For $\lambda_1, \lambda_2, \mu_1, \mu_2, \rho_1, \rho_2 \in L^\infty_+(\Omega)$, we have
\begin{equation}\label{monotonicity_corol}
\lambda_1\leq \lambda_2,\quad \mu_1\leq \mu_2 \text{ and } \rho_1\leq \rho_2 \quad \text{  implies } \quad \Lambda_{\lambda_1\mu_1,\rho_1}\geq \Lambda_{\lambda_2\mu_2,\rho_2},
\end{equation}
in the sense of quadratic forms.
\end{corollary}
In the following theorem,  we prove that we can control the
energy terms appearing in the monotonicity relation \eqref{mono1}. We will first state
the result and prove it using a functional analytic relation between operator norms and the ranges of their
adjoints.
\begin{theorem}[Localized potentials]
\label{thm:locpot}
Let $\lambda, \mu, \rho$ satisfying Hypothesis 1  and $\mathcal D_1, \mathcal D_2\Subset\Omega $ be   non-empty  open sets    such that $\overline{\mathcal D_1 }\cap\overline{\mathcal D_2}=\emptyset$ and   $\Omega\setminus(\overline{\mathcal D_1}\cup\overline{\mathcal D_2})$  is  connected.
\begin{itemize}
\item[i)] There exists a sequence
$(g_n)_{n\in \mathbb{N}}\subset L^2(\Gamma_\textup{N},\R^d)$ such that the corresponding solutions
$(u^{(g_n)})_{n\in \mathbb{N}}$ of (\ref{direct}) fulfill 
\begin{equation}
\label{localized_u1}
\lim_{n\to \infty}\int_{\mathcal D_1}\Vert \nabla \cdot u^{(g_n)}\Vert^2 \,dx=\infty,
\end{equation}
\begin{equation}
\label{localized_u2}
\lim_{n\to \infty}\int_{\mathcal D_1}\Vert \nabla^s u^{(g_n)}\Vert^2_F\,dx=\infty,
\end{equation}
\begin{equation}
\label{localized_u3}
\lim_{n\to \infty}\int_{\mathcal D_2}\Vert \nabla\cdot u^{(g_n)}\Vert^2\,dx=0.
\end{equation}
\item[ii)] There exists a sequence
$(\tilde g_n)_{n\in \mathbb{N}}\subset L^2(\Gamma_\textup{N},\R^d)$ such that the corresponding solutions
$(u^{(\tilde g_n)})_{n\in \mathbb{N}}$ of (\ref{direct}) fulfill 
\begin{equation}
\label{localized_u4}
\lim_{n\to \infty}\int_{\mathcal D_1}\Vert  u^{(\tilde g_n)}\Vert^2\,dx=\infty.
\end{equation}
\begin{equation}
\label{localized_u5}
\lim_{n\to \infty}\int_{\mathcal D_2}\Vert u^{(\tilde g_n)}\Vert^2\,dx=0,
\end{equation}
\end{itemize}

\end{theorem}
Before proving Theorem  \ref{thm:locpot},  we state the following lemmas.
Define the virtual  measurement operators 
 $T_j, Z_j$,  $j=1,2$, by
\[
T_j :  L^2(\mathcal D_j,\R)\rightarrow  L^2(\Gamma_\textup{N},\R^d), \quad  F\mapsto  v|_{\Gamma_\textup{N}},
\]
\[
Z_j :  L^2(\mathcal D_j,\R^d)\rightarrow  L^2(\Gamma_\textup{N},\R^d), \quad  G\mapsto  w|_{\Gamma_\textup{N}},
\]

 where $v, w\in H^1(\Omega,\R^d)$ solve 
 \begin{equation}\label{dual_T} 
 \int_\Omega\lambda(\nabla\cdot v)(\nabla\cdot \varphi)+2\mu\nabla^s v:\nabla^s \varphi\,dx+\int_\Omega \rho v\cdot \varphi\,dx=\int_{\mathcal D_j}  F\cdot (\nabla\cdot\varphi)\,dx \quad \text{ for all } \varphi\in \V, 
 \end{equation}
 \begin{equation}\label{dual_Z} 
  \int_\Omega\lambda(\nabla\cdot w)(\nabla\cdot \psi)+2\mu\nabla^s w:\nabla^s \psi\,dx+\int_\Omega \rho w\cdot \psi\,dx=\int_{\mathcal D_j}  G\cdot \psi\,dx \quad \text{ for all } \psi\in \mathcal{V}.
 \end{equation}
 
\begin{lemma}\label{dual}
The dual operators $T^*_j,  Z^*_j,  j=1,2$,  are given by
 \begin{equation}
\label{dual1}
T^*_j : L^2(\Gamma_\textup{N},\R^d)  \rightarrow  L^2(\mathcal D_j,\R):  g \mapsto T^*_jg= \nabla\cdot u|_{\mathcal D_j}, 
\end{equation}
\begin{equation}
\label{dual2}
Z^*_j : L^2(\Gamma_\textup{N},\R^d)  \rightarrow   L^2(\mathcal D_j,\R^d): g\mapsto Z^*_j g= u|_{D_j},
\end{equation}
where $u$ is the  unique solution of   (\ref{direct}) with respect to the Neumann boundary data $g$.
\end{lemma}
\begin{proof}
To prove \eqref{dual1}, let $F\in L^2(\Omega,\R^d)$, $g\in L^2(\Gamma_N,\R^d)$, $u$, $v\in \V$ solve (\ref{direct}) and (\ref{dual_T}), respectively. Then,
\[
\begin{aligned}
\int_{\Omega}F  T^{*}_j g\,dx 
&= \int_{\Gamma_\textup{N}}g\cdot T_j F\,ds
=\int_{\Omega}\lambda(\nabla\cdot v)(\nabla\cdot u)+2\mu\nabla^s v:\nabla^s u\,dx+ \int_\Omega  \rho v\cdot u\,dx\\
&=\int_{\mathcal D_j} F (\nabla\cdot u)\,dx.
\end{aligned}
\]
This implies that   $T^{*}_j g= \nabla\cdot u\vert_{\mathcal D_j}$.\\
\noindent
To prove \eqref{dual2}, let $G\in L^2(\Omega, \R^{d})$, $g\in L^2(\Gamma_N,\R^d)$,  $u$, $w\in \V$ solve (\ref{direct}) and (\ref{dual_Z}), respectively. Then,
\[
\begin{aligned}
\int_{\Omega}G :  Z^{*}_j g \,dx
& = \int_{\Gamma_N}g\cdot Z_j G\,ds
 =\int_{\Omega}\lambda(\nabla\cdot w)(\nabla\cdot u)+2\mu \nabla^s w :\nabla^s u\,dx+\int_\Omega \rho w\cdot u\,dx\\
 &=\int_{ \mathcal D_j}G\cdot u\,dx,
\end{aligned}
\]
which implies that   $Z^{*}_j g= u\vert_{\mathcal D_j}$.
\end{proof}
 Next, we  show some properties of  the ranges  $\mathcal{R}(T_j)$, and $\mathcal{R}(Z_j),   j=1,2$.
 \begin{lemma} 
 \label{range1}
 The ranges  $\mathcal{R}(T_j),  \mathcal{R}(Z_j, j=1,2$,  are  dense in $L^2(\Gamma_N,\R^d)$ and 
\[
\mathcal{R}(T_1)\cap  \mathcal{R}(T_2)=\lbrace 0\rbrace,
\]
\[
\mathcal{R}(Z_1)\cap  \mathcal{R}(Z_2)=\lbrace 0\rbrace.
\]
\end{lemma}
\begin{proof}
 The proof is based on the unique continuation principle for Cauchy data  \cite{lin2010quantitative}. 
\end{proof}

\begin{proof}[Proof of Theorem \ref{thm:locpot}]
Lemma \ref{range1}   implies  the range non inclusions  $\mathcal{R}(T_1)\not\subseteq \mathcal{R}(T_2)$ and  $\mathcal{R}(Z_1)\not\subseteq \mathcal{R}(Z_2)$.
Using \cite[Corollary 2.6]{gebauer2008localized} together with Lemma \ref{dual}, it follows that there exist  sequences
$(g_n)_{n\in \mathbb{N}}, (\tilde g_n)_{n\in \mathbb{N}}\subset L^2(\Gamma_\textup{N},\R^d)$  such  that  the corresponding  solutions  $(u^{(g_n)})_{n\in\mathbb{N}}, (u^{(\tilde g_n)})_{n\in\mathbb{N}}
\subset \mathcal{V}$ fulfill
\[
\lim_{n\to \infty}\Vert T_1^* g_n\Vert^2_{L^2(\mathcal D_1,\R^{d})}=\lim_{n\to \infty}\int_{\mathcal D_1}\Vert \nabla\cdot u^{(g_n)}\Vert^2\,dx
=\infty,
\]
\[
\lim_{n\to \infty}\Vert T_2^* g_n\Vert^2_{L^2(\mathcal D_2,\R^{d})}=\lim_{n\to \infty}\int_{\mathcal D_2}\Vert \nabla \cdot u^{(g_n)}\Vert^2\,dx
=0,
\]

\noindent
and
\[
\lim_{n\to \infty}\Vert Z_1^* \tilde g_n\Vert^2_{L^2(\mathcal D_1,\R^d)}=\lim_{n\to \infty}\int_{\mathcal D_1}\Vert u^{(\tilde g_n)}\Vert^2\,dx
=\infty,
\]
\[
\lim_{n\to \infty}\Vert Z_2^* \tilde g_n\Vert^2_{L^2(\mathcal D_2,\R^{d\times d})}=\lim_{n\to \infty}\int_{\mathcal D_2}\Vert u^{(\tilde g_n)}\Vert^2\,dx=0.
\]
This proves \eqref{localized_u1}, \eqref{localized_u3}-\eqref{localized_u5}.   \eqref{localized_u2}
follows from  \eqref{localized_u1} and  the proof is completed.
\end{proof}
\subsection{ Lipschitz stability}
 Let $\mathcal{P}$ be a finite dimensional subset of  $\hat C(\Omega)\times C^{0,1}(\Omega)\times \hat C(\Omega)$ where  $\hat C(\Omega)$ is the set of piecewise  cintinuous functions.   We consider six constants  $0<a\leq b$,   $0< c \leq d$  and $0< e \leq f$,
  which are the lower and upper bounds of the   parameters  $\lambda,\mu$ and $\rho$,  and define the sets
\[
 \mathcal{P}_{[a,b]\times[c, d]\times[e,f]}=\left\{ (\lambda,\mu, \rho)\in \mathcal{P}:  \quad  a\leq\lambda(x)\leq b,  \quad  c\leq \mu\leq d, \quad e\leq \rho\leq f \textrm{ for all } x\in \Omega \right\}.
\]
 $\mathcal{E}:= \mathcal{E}_+ \cup \mathcal{E}_-$, with
\[
 \begin{aligned}
 \mathcal{E}_+ &=\left\{ (\zeta_1,\zeta_2, \zeta_3)\in \rm{span }\,\mathcal{P}:  \quad  \zeta_1,\zeta_2,\zeta_3\geq 0 \quad \text{  and }\quad 
 \Vert (\zeta_1, \zeta_2, \zeta_3)\Vert_\Delta= 1 \right\},\cr
 \mathcal{E}_-&=\left\{ (\zeta_1,\zeta_2, \zeta_3)\in \rm{span }\,\mathcal{P}:  \quad   \zeta_1, \zeta_2, \zeta_3\leq 0\quad \text{  and }\quad 
\Vert (\zeta_1, \zeta_2, \zeta_3)\Vert_\Delta = 1 \right\},
 \end{aligned}
\]
where  
\begin{equation}
\label{norm_inf}
\Vert (\zeta_1, \zeta_2, \zeta_3)\Vert_\Delta:= \max\left(\Vert \zeta_1\Vert_{L^\infty(\Omega)}, \Vert\zeta_2\Vert_{L^\infty(\Omega)},
\Vert\zeta_3\Vert_{L^\infty(\Omega)}\right).
\end{equation}
In the following main result of this paper, the domain $\Omega$, the finite-dimensional subset $\mathcal{P}
$ and the bounds $0< a\leq b$, $0< c \leq  d$  and  $0< e \leq  f$   are fixed, and the constant in the Lipschitz stability result will depend on them.
\begin{theorem}[Lipschitz stability]
\label{stability1}
 There exists a positive constant $C>0$ such that for all 
$(\lambda_1,\mu_1, \rho_1), (\lambda_2,\mu_2, \rho_2) \in \mathcal{P}_{[a, b]\times[c , d]\times[e,f]}$ with either
\[
\begin{aligned}
&(i)\quad \lambda_1\leq \lambda_2,\quad  \mu_1\leq \mu_2 \textrm{ and } \rho_1\leq \rho_2\quad \textrm{or},\\
&(ii)\quad \lambda_1\geq \lambda_2,\quad \mu_1\geq \mu_2 \textrm{ and }  \rho _1\geq \rho_2,
\end{aligned}
\]
\noindent
we have 
\begin{equation}
\label{stab-est1}
\Vert (\lambda_1-\lambda_2,\mu_1-\mu_2, \rho_1-\rho_2)\Vert_\Delta \leq   C\| \Lambda_{\lambda_1,\mu_1, \rho_1}-\Lambda_{\lambda_2,\mu_2,\rho_2} \|_*.
\end{equation}
Here  $\Vert.\Vert_*$ is the operator  norm of ${\mathcal L(L^2(\Gamma_\textup{N},\R^d))}$ and $\Vert \cdot\Vert_\Delta$  is defined in \eqref{norm_inf}. 
\end{theorem}
\subsubsection{Proof of Theorem \ref{stability1}}
In this subsection, we prove Theorem \ref{stability1} using    the monotonicity relations in Lemma \ref{mono}  and the results of localized potentials derived in Theorem \ref{thm:locpot}
\begin{lemma}
 For $(\lambda_1,\mu_1,\rho_1)\neq (\lambda_2,\mu_2,\rho_2)$, we have 
\begin{equation}\label{estim}
\begin{aligned}
 &\frac{\Vert \Lambda_{\lambda_2,\mu_2,\rho_2}-\Lambda_{\lambda_1,\mu_1,\rho_1} {\Vert_*}}
{\Vert (\lambda_1-\lambda_2,\mu_1-\mu_2, \rho_1-\rho_2)\Vert_\Delta}\\
& \geq\inf_{\substack{(\zeta_1, \zeta_2,\zeta_3)\in \mathcal{E}\cr (\kappa_1,\tau_1,\eta_1),(\kappa_2,\tau_2,\eta_2)\in\mathcal{P}_{[a , b]\times[c, d]\times[e,f]}}} \sup_{\| g \|=1} \Phi\left(g,(\zeta_1,\zeta_2,\zeta_3),(\kappa_1,\tau_1, \eta_1),(\kappa_2,\tau_2, \eta_2)\right),
\end{aligned}
\end{equation}
where 
\[
\begin{aligned}
& \Phi\left(g,(\zeta_1,\zeta_2,\zeta_3),(\kappa_1,\tau_1,\eta_1),(\kappa_2,\tau_2,\eta_2)\right)\\
& :=\max \left(\Psi\left(g,(\zeta_1,\zeta_2,\zeta_3),(\kappa_1,\tau_1,\eta_1)\right), \Psi\left(g,(-\zeta_1,-\zeta_2, -\zeta_3),(\kappa_2,\tau_2,\eta_2)\right)  \right),
\end{aligned}
\]
with
\[
\Psi\left(g,(\alpha,\beta,\gamma),(\kappa,\tau,\eta )\right):=
\int_{\Omega}\alpha \nabla\cdot u_{(\kappa,\tau,\eta)}^{g}\,dx+
2\int_{\Omega} \beta \vert\nabla^s u_{(\kappa,\tau,\eta)}^{g}\vert^2\,dx+
\int_\Omega \gamma\vert u^{g}_{(\kappa,\tau,\eta)}\vert^2\,dx.
\]

\end{lemma}
\begin{proof}
Since the Neumann-to-Dirichlet  operator  $\Lambda_{\lambda,\mu,\rho}$  is  self-adjoint, we   obtain   
\[
\begin{aligned}
& \Vert\Lambda_{\lambda_2,\mu_2,\rho_2}-\Lambda_{\lambda_1,\mu_1,\rho_1}\Vert_*\\
&=  \sup_{\Vert g\Vert=1} \vert  \langle g, \left(\Lambda_{\lambda_2,\mu_2,\rho_2}-\Lambda_{\lambda_1,\mu_1,\rho_1}\right) g\rangle\vert \cr
&=  \sup_{\Vert g\Vert=1} \max\left\{ \langle g, \left(\Lambda_{\lambda_2,\mu_2,\rho_2}-\Lambda_{\lambda_1,\mu_1,\rho_1}\right) g\rangle,
\langle g, \left(\Lambda_{\lambda_1,\mu_1,\rho_1}-\Lambda_{\lambda_2,\mu_2,\rho_2}\right) g\rangle\right\}.
\end{aligned}
\]
\noindent
From the  monotonicity relation (\ref{mono1}) in Lemma \ref{mono}, we   obtain for  $g\in L^2(\Gamma_\textup{N},\R^d)$
\noindent
\begin{equation}\label{estim_01}
\begin{aligned}
& \langle g,\left(\Lambda_{\lambda_2,\mu_2,\rho_2}-\Lambda_{\lambda_1,\mu_1,\rho_1}\right) g\rangle\\ 
&\geq \int_\Omega(\lambda_1-\lambda_2)\Vert\nabla\cdot u_{(\lambda_1,\mu_1,\rho_1)}^{g}\Vert^2\,dx
+ 2\int_\Omega(\mu_1-\mu_2)\Vert\nabla^s u_{(\lambda_1,\mu_1,\rho_1)}^{g}\Vert^2_F\,dx\\
&+\int_\Omega(\rho_1-\rho_2) \Vert u^{g}_{(\lambda_1,\mu_1,\rho_1)}\Vert^2\,dx,
\end{aligned}
\end{equation}
and
\begin{equation}\label{estim_02}
\begin{aligned}
& \langle g,\left(\Lambda_{\lambda_1,\mu_1,\rho_1}-\Lambda_{\lambda_2,\mu_2,\rho_2}\right) g\rangle\\ 
&\geq \int_\Omega(\lambda_2-\lambda_1)\Vert\nabla\cdot u_{(\lambda_2,\mu_2,\rho_2)}^{g}\Vert^2\,dx
+ 2\int_\Omega(\mu_2-\mu_1)\Vert\nabla^s u_{(\lambda_2,\mu_2,\rho_2)}^{g}\Vert^2_F\,dx\\
&+\int_\Omega(\rho_2-\rho_1) \Vert u^{g}_{(\lambda_2,\mu_2,\rho_2)}\Vert^2\,dx,
\end{aligned}
\end{equation}
where $u_{\lambda_1,\mu_1,\rho_1}^{g},u_{\lambda_2,\mu_2,\rho_2}^{g}\in  \V$ denote the solutions of (\ref{direct}) with Neumann data $g$ and parameters  $(\lambda_1,\mu_1, \rho_1)$ and $(\lambda_2,\mu_2,\rho_2)$, respectively.
 Based on the estimates (\ref{estim_01}) and (\ref{estim_02}), we obtain for $(\lambda_1,\mu_1,\rho_1)\neq (\lambda_2,\mu_2,\rho_2)$
\begin{equation}
\label{estim_03}
 \frac{\Vert \Lambda_{\lambda_2,\mu_2,\rho_2}-\Lambda_{\lambda_1,\mu_1,\rho_1} {\Vert_*}}{\Vert (\lambda_1-\lambda_2,\mu_1-\mu_2, \rho_1-\rho_2)\Vert_\Delta}
\geq  
 \sup_{\Vert g\Vert=1}\Phi\left(g,\Theta_1,\Theta_2,\Theta_3,(\lambda_1,\mu_1,\rho_1),(\lambda_2,\mu_2, \rho_2)\right),
\end{equation}
with
\[
\Theta_1=\Theta_1(\lambda_1,\lambda_2,\mu_1,\mu_2,\rho_1,\rho_2):=\frac{\lambda_1-\lambda_2}{\Vert (\lambda_1-\lambda_2,\mu_1-\mu_2, \rho_1-\rho_2)\Vert_\Delta},
\]
\[
\Theta_2=\Theta_2(\lambda_1,\lambda_2,\mu_1,\mu_2,\rho_1,\rho_2):=\frac{\mu_1-\mu_2}{\Vert (\lambda_1-\lambda_2,\mu_1-\mu_2, \rho_1-\rho_2)\Vert_\Delta},
\]
\[
\Theta_3=\Theta_3(\lambda_1,\lambda_2,\mu_1,\mu_2,\rho_1,\rho_2):=\frac{\rho_1-\rho_2}{\Vert (\lambda_1-\lambda_2,\mu_1-\mu_2, \rho_1-\rho_2)\Vert_\Delta}.
\]
Then,  using that either assumption $(i)$ or assumption $(ii)$ is fulfilled, we can rewrite (\ref{estim_03}) as
\[
\begin{aligned}
& \frac{\Vert \Lambda_{\lambda_2,\mu_2,\rho_2}-\Lambda_{\lambda_1,\mu_1,\rho_1} {\Vert_*}}{\Vert (\lambda_1-\lambda_2,\mu_1-\mu_2, \rho_1-\rho_2)\Vert_\Delta}\\
&\geq\inf_{\substack{(\zeta_1, \zeta_2,\zeta_3)\in \mathcal{E}\cr (\kappa_1,\tau_1,\eta_1),(\kappa_2,\tau_2,\eta_2)\in\mathcal{P}_{[a , b]\times[c, d]\times[e,f]}}} \sup_{\| g \|=1} \Phi\left(g,(\zeta_1,\zeta_2,\zeta_3),(\kappa_1,\tau_1,\eta_1),(\kappa_2,\tau_2,\eta_2)\right),
\end{aligned}
\]
and the proof is completed.
\end{proof}
The assertion of Theorem \ref{stability1} follows if we can show that the right-hand side of (\ref{estim})
is  strictly positive. Since $\Phi$ is continuous,  we can conclude that the function
\[
\left((\zeta_1,\zeta_2,\zeta_3),(\kappa_1,\tau_1,\eta_1),(\kappa_2,\tau_2,\eta_2)\right)\mapsto \sup_{\| g \|=1} \Phi\left(g,(\zeta_1,\zeta_2,\zeta_3),(\kappa_1,\tau_1,\eta_1),(\kappa_2,\tau_2,\eta_2)\right),
\]
is lower semi-continuous, so that it attains its minimum on  the compact set \\
$\mathcal{E}\times \mathcal{P}_{[a, b]\times[c, d]\times[e,f]}  \times \mathcal{P}_{[a, b]\times[c, d]\times[e,f]} $.
Hence, to prove Theorem \ref{stability1}, it suffices to show that
\begin{equation}\label{estim_04}
\sup_{\| g \|=1} \Phi\left(g,(\zeta_1,\zeta_2,\zeta_3),(\kappa_1,\tau_1,\eta_1),(\kappa_2,\tau_2,\eta_2)\right)>0, 
\end{equation}
for all $
\left((\zeta_1,\zeta_2,\zeta_3),(\kappa_1,\tau_1,\eta_1),(\kappa_2,\tau_2,\eta_2)\right)\in \mathcal{E}\times \mathcal{P}_{[a , b]\times[c, d]\times[e,f]}  \times \mathcal{P}_{[a, b]\times[c, d]\times[e,f]}.$

 In order to prove that (\ref{estim_04}) holds true, let $\left((\zeta_1,\zeta_2,\zeta_3),(\kappa_1,\tau_1,\eta_1),(\kappa_2,\tau_2,\eta_2)\right)\in \mathcal{E}\times \mathcal{P}_{[a, b]\times[c, d]\times[e,f]}  \times \mathcal{P}_{[a, b]\times[c, d]\times[e,f]}$.
\\
We first treat the case that $(\zeta_1,\zeta_2,\zeta_3)\in \mathcal{E}_+ $. Then,  there exist  a nom-empty open subset $\mathcal D_1\Subset \Omega$ and  a constant $0<\delta<1$,  such that either
\[
\begin{aligned}
&\text{(i)}\   \zeta_1|_{\mathcal D_1}\geq \delta,  \text{ and } \zeta_2, \zeta_3\geq 0, \textrm{ or }\cr
&\text{(ii)}\   \zeta_2|_{\mathcal D_1}\geq \delta,  \text{ and } \zeta_1,\zeta_3\geq 0, \textrm{ or }\cr
&\text{(iii)}\   \zeta_3|_{\mathcal D_1}\geq \delta,  \text{ and } \zeta_1,\zeta_2\geq 0.
\end{aligned}
\]
In case (i),  we use the localized potentials results  from Theorem \ref{thm:locpot},   to obtain  an open subset $\mathcal D_2\Subset \Omega$
with $\overline{\mathcal D_1}\cap \overline{\mathcal D_2}=\emptyset$  and a  boundary load
$\tilde g\in L^2(\Gamma_N,\mathbb{R}^d)$ with 
\begin{equation}\label{estim_loc_pot}
\int_{\mathcal D_1}\Vert \nabla\cdot u^{\tilde g}_{(\kappa_1,\tau_1,\eta_1)}\Vert^2\,dx  \geq \frac{1}{\delta }
\end{equation}
 This leads to
\[
 \begin{aligned}
&\Psi\left(\tilde g,(\zeta_1,\zeta_2,\zeta_3),(\kappa_1,\tau_1,\eta_1)\right)\\
&= \int_\Omega \zeta_1\Vert\nabla^s\cdot u^{\tilde g}_{(\kappa_1,\tau_1,\eta_1)}\Vert^2\,dx
+2\int_\Omega \zeta_2\Vert\nabla^s u^{\tilde g}_{(\kappa_1,\tau_1,\eta_1)}\Vert^2_F\,dx+
\int_\Omega\zeta_3 \Vert u^{\tilde g}_{(\kappa_1,\tau_1,\eta_1)}\Vert^2\, dx \cr
& \geq \int_{\mathcal D_1} \zeta_1 \Vert\nabla\cdot u^{\tilde g}_{(\kappa_1,\tau_1,\eta_1)}\Vert^2\,dx
 \geq \delta\int_{\mathcal D_1}  \Vert\nabla\cdot u^{\tilde g}_{(\kappa_1,\tau_1,\eta_1)}\Vert^2\,dx \geq 1.
 \end{aligned}
\]
In case $(ii)$, we use the localized potentials results  from Theorem \ref{thm:locpot} to obtain  an open subset $\mathcal D_2\Subset \Omega$
with $\overline{\mathcal D_1}\cap \overline{\mathcal D_2}=\emptyset$  and a  boundary load
$\tilde g\in L^2(\Gamma_N,\mathbb{R}^d)$ with 
\[
\int_{\mathcal D_1}  \Vert\nabla^s u^{\tilde g}_{(\kappa_1,\tau_1,\eta_1)}\Vert_F^2\,dx  \geq \frac{1}{2\delta}.
\]
We get 
\[
 \begin{aligned}
&\Psi\left(\tilde g,(\zeta_1,\zeta_2,\zeta_3),(\kappa_1,\tau_1,\eta_1)\right)\\
&\geq  \int_\Omega \zeta_1\Vert\nabla\cdot u^{\tilde g}_{(\kappa_1,\tau_1,\eta_1)}\Vert^2\,dx
+2\int_\Omega \zeta_2\Vert\nabla^s u^{\tilde g}_{(\kappa_1,\tau_1, \eta_1)}\Vert^2_F\,dx+
\int_\Omega\zeta_3 \Vert u^{\tilde g}_{(\kappa_1,\tau_1,\eta_1)}\Vert^2\, dx  \cr
&\geq   2\int_{\mathcal D_1} \zeta_2 \Vert \nabla^s u^{\tilde g}_{(\kappa_1,\tau_1,\eta_1)}\Vert^2_F\, dx
 \geq 2\delta \int_{\mathcal D_1} \Vert \nabla^s u^{\tilde g}_{(\kappa_1,\tau_1,\eta_1)}\Vert^2_F\, dx  \geq 1.
 \end{aligned}
\]
In case $(iii)$, we use the localized potentials results  from Theorem \ref{thm:locpot} to obtain  an open subset $\mathcal D_2\Subset \Omega$
with $\overline{\mathcal D_1}\cap \overline{\mathcal D_2}=\emptyset$  and a  boundary load
$\tilde g\in L^2(\Gamma_N,\mathbb{R}^d)$ with 
\[
\int_{\mathcal D_1}  \Vert u^{\tilde g}_{(\kappa_1,\tau_1,\eta_1)}\Vert^2\,dx  \geq \frac{1}{\delta}.
\]
We obtain 
\[
 \begin{aligned}
&\Psi\left(\tilde g,(\zeta_1,\zeta_2,\zeta_3),(\kappa_1,\tau_1,\eta_1)\right)\\
& = \int_{\mathcal D_1} \zeta_3 \Vert u^{\tilde g}_{(\kappa_1,\tau_1,\eta_1)}\Vert^2\,dx
 \geq \delta\int_{\mathcal D_1}  \Vert u^{\tilde g}_{(\kappa_1,\tau_1,\eta_1)}\Vert^2\,dx \geq 1.
\end{aligned}
\]
In all three cases, we can find  $\tilde g\in L^2(\Gamma_N,\mathbb{R}^d)$  such that
\[
 \Psi(\tilde g,(\zeta_1,\zeta_2,\zeta_3),(\kappa_1,\tau_1,\eta_1),(\kappa_2,\tau_2,\eta_2))>0.
\]
For  $(\zeta_1,\zeta_2,\zeta_3) \in  \mathcal{E}_-$, we can analogously use a localized potentials sequence for $(\kappa_2,\tau_2,\eta_2)$, and we can find  $\tilde g\in L^2(\Gamma_N,\mathbb{R}^d)$
such that 
\[
 \Psi(\tilde g,(-\zeta_1,-\zeta_2,-\zeta_3),(\kappa_1,\tau_1,\eta_1),(\kappa_2,\tau_2,\eta_2))>0.
\]
We conclude that
\[
 \begin{aligned}
&\sup_{\| g \|=1} \Phi(g,(\zeta_1,\zeta_2,\zeta_3),(\kappa_1,\tau_1,\eta_1),(\kappa_2,\tau_2,\eta_2))
\geq  \Phi\left(\frac{\tilde g}{\Vert \tilde g\Vert},(\zeta_1,\zeta_2,\zeta_3),(\kappa_1,\tau_1,\eta_1),(\kappa_2,\tau_2,\eta_2)\right)\cr
&=\frac{1}{\Vert\tilde g\Vert^2} \Phi(\tilde g,(\zeta_1,\zeta_2,\zeta_3),(\kappa_1,\tau_1,\eta_1),(\kappa_2,\tau_2,\eta_2))>0,
\end{aligned}
\]
 which completes  the proof of  Theorem \ref{stability1}.
 \section{Conclusion}
In this study, we have tackled the inverse problem of recovering the Lamé parameters $\lambda,\mu$ and  the density  $\rho$ from the Neumann-to-Dirichlet map $\Lambda_{\lambda,\mu, \rho}$. Specifically, when the Lamé parameters $\lambda, \mu$ are  known, we have established non constructiv Lipschitz stability result for recovering the density profile $\rho$. For the scenario involving piecewise constant parameters in a given partition, constructive Lipschitz stability result is derived.
Additionally, under a definiteness assumption, we have demonstrated a Lipschitz stability result for simultaneously recovering the coefficients $\lambda, \mu$ and $\rho$. In both cases, we assume that the coefficients have known upper and lower bounds and belong to a finite-dimensional subspace.

Our proofs primarily rely on establishing a monotonicity relation between the parameters $\lambda,\mu, \rho$ and the Neumann-to-Dirichlet map combined  with the utilization of localized potentials. \\

 \noindent
{\bf Conflict of interest}:The authors declare no conflict of interests.\\
{\bf Data availability statement}:  Not applicable. \\
{ \bf Funding statement}: This work is supported by the Ministry of Higher Education and Scientific Research (Tunisia).

\bibliography{biblio}

\end{document}